\documentclass[a4paper,12pt,reqno]{amsart}
\usepackage{amsmath,amsthm,amssymb}
\usepackage{hyperref}

\hypersetup{unicode,bookmarksnumbered=true,hidelinks,final}

\theoremstyle{definition}
\newtheorem{theorem}{Theorem}[section]
\newtheorem{definition}[theorem]{Definition}
\newtheorem{proposition}[theorem]{Proposition}
\newtheorem{lemma}[theorem]{Lemma}
\newtheorem{corollary}[theorem]{Corollary}
\newtheorem{remark}[theorem]{Remark}

\numberwithin{equation}{section}
\makeatletter
\let\c@equation\c@theorem
\makeatother

\begin{document}

\title[]{The Canonical Metric on Holomorphic Pairs\\ over Compact Non-K\"{a}hler Manifolds}
\author{Ryoma Saito}
\date{}
\address{Department of Mathematics, Institute of Science Tokyo, 2-12-1, O-okayama, Meguro, 152-8551, Japan}
\email{saito.r.ai@m.titech.ac.jp, saito.ryoma.7195@gmail.com}


\begin{abstract}
  In this paper, we prove the solvability of the vortex equation on a holomorphic vector bundle over a compact Hermitian manifold using the continuity method, and show the Kobayashi-Hitchin correspondence for holomorphic pairs. This work extends Bradlow's Kobayashi-Hitchin correspondence over compact K\"{a}hler manifolds to compact non-K\"{a}hler manifolds.
\end{abstract}

\maketitle

\section{Introduction}
For holomorphic vector bundles over compact complex manifolds, the equivalence between the stability of vector bundles and the existence of canonical metrics is known as the Kobayashi-Hitchin correspondence. The canonical metric here is a solution of the curvature equation, specifically the Hermitian-Einstein metric:
\begin{gather*}
i\Lambda_g F_h-\lambda\mathrm{id}_E=0,~\lambda\in\mathbb{R}.
\end{gather*}

The proof of the Kobayashi-Hitchin correspondence developed through several key contributions. Kobayashi \cite{MR556302, MR664562, MR909698} and Lübke \cite{MR701206} first showed that the stability of vector bundles follows from the existence of a canonical metric. Later, Donaldson \cite{MR710055, MR765366, MR885784} and Uhlenbeck-Yau \cite{MR861491} proved that a stable vector bundle admits a canonical metric by taking different differential geometric approaches. Donaldson used the heat flow method, while Uhlenbeck-Yau used the continuity method.

After the establishment of this Kobayashi-Hitchin correspondence, the following extensions have been studied:
\begin{enumerate}
  \item Extension to compact Gauduchon manifolds (which are, by definition, $n$-dimensional compact complex manifolds with a Hermitian metric $\omega_g$ satisfying $\partial\bar{\partial}\omega_g^{n-1}=0$).
  \item Extension to vector bundles with additional structures (Higgs fields or holomorphic sections) over compact K\"{a}hler manifolds.
\end{enumerate}
Li-Yau \cite{MR915839} demonstrated the first extension using the continuity method. For the second extension, the case with Higgs fields was developed by Simpson \cite{MR944577}, and the case with holomorphic sections was investigated by Bradlow \cite{MR1085139}. In particular, Bradlow defined the notion of the canonical metric as a solution of the vortex equation and the stability for holomorphic pairs $(E,\phi)$ (consisting of a holomorphic vector bundle and a holomorphic section) over compact K\"{a}hler manifolds. He then proved the Kobayashi-Hitchin correspondence for holomorphic pairs using the heat flow method. The vortex equation is given as follows, and its solution $h$ is called a $\tau$-Hermitian-Einstein metric.
\begin{gather*}
  i\Lambda_g F_h+\frac{1}{2}\phi\otimes\phi^*_h-\frac{\tau}{2}\mathrm{id}_E=0,~\tau\in\mathbb{R}
\end{gather*}
His motivation for considering holomorphic pairs came from the study of the Higgs bundles and also the similarity between the vortex equation and the equation giving critical points of the Yang-Mills-Higgs functional \cite{MR1086749}. In these previous studies, the base manifold is always assumed to be K\"{a}hler.

In this paper, we study the Kobayashi-Hitchin correspondence for vector bundles with additional structures over compact Gauduchon manifolds, generalizing Bradlow's theorem to non-K\"{a}hler manifolds. The main theorem is as follows.
\begin{theorem}\label{main-theorem}
  If a holomorphic pair $(E,\phi)$ over an $n$-dimensional compact Gauduchon manifold $(X,\omega_g)$ is $\tau$-stable, then it admits a unique $\tau$-Hermitian-Einstein metric.
\end{theorem}

It can be shown by standard arguments that the existence of a canonical metric implies the stability of the vector bundle. Therefore, we obtain the following Kobayashi-Hitchin correspondence.

\begin{theorem}
  For a holomorphic pair $(E,\phi)$ over an $n$-dimensional compact Gauduchon manifold $(X,\omega_g)$, the following (1) and (2) are equivalent.
  \begin{itemize}
    \item [(1)] $(E,\phi)$ admits a $\tau$-Hermitian-Einstein metric.
    \item [(2)] $(E,\phi)$ satisfies either (a) or (b).
    \begin{itemize}
      \item [(a)] $(E,\phi)$ is $\tau$-stable,
      \item [(b)] $E$ decomposes holomorphically as $E=E_\phi\oplus E^{\prime\prime}$, where $\phi\in E_\phi$ is $\tau$-stable, and $E^{\prime\prime}$ is polystable with $\mu(E^{\prime\prime})=\frac{\tau}{4\pi}\mathrm{Vol}_g(X)$.
    \end{itemize}
  \end{itemize}
\end{theorem}

This theorem can be regarded either as Li-Yau's Kobayashi-Hitchin correspondence with an added holomorphic section, or as an extension of Bradlow's Kobayashi-Hitchin correspondence to compact Gauduchon manifolds.

For the proof of Theorem \ref{main-theorem}, it would be natural to follow either Bradlow's or Li-Yau's proof methods. However, neither method can be directly applied, and difficulties arise. Specifically, Bradlow's heat flow method uses the Donaldson functional defined on K\"{a}hler manifolds and essentially relies on the K\"{a}hler property, making it difficult to apply to non-K\"{a}hler manifolds. For this reason, we use Li-Yau's continuity method to construct the Hermitian-Einstein metric. However, in their construction process, they use the simplicity of vector bundles, which follows from the stability of vector bundles. The simplicity of vector bundles does not follow from the $\tau$-stability of vector bundles, so we need to remove this assumption in the construction of the $\tau$-Hermitian-Einstein metric. In this paper, we prove it without using the simplicity.

In addition to the extensions of the Kobayashi-Hitchin correspondence mentioned above, the extension to non-compact manifolds has also been studied. In particular, Wang-Zhang \cite{MR4309984} proved the existence of a $\tau$-Hermitian-Einstein metric in a $\tau$-stable holomorphic pair over a non-compact manifold satisfying suitable assumptions. It should be noted that Theorem \ref{main-theorem} is not included in their result.
\bigskip

The structure of this paper is as follows. In Chapter 2, we recall some definitions about holomorphic pairs. The proof of the main theorem (Theorem \ref{main-theorem}) is presented in Chapter 3. In particular, we develop arguments that do not rely on the simplicity of vector bundles when proving the solvability of the perturbed equations (Section \ref{sol-perturb}) and the existence of limits for their solutions (Section \ref{sol-vol}). In Chapter 4, we discuss the Higgs bundle version of the main theorem.

\section{Preliminaries}
On an $n$-dimensional compact Hermitian manifold $(X,\omega_g)$, we always have a metric satisfying $\partial\bar{\partial}\omega_g^{n-1}=0$, which is called the Gauduchon metric. Let $E=(E,\bar{\partial}_E)$ be a rank $r$ holomorphic vector bundle over an $n$-dimensional compact Gauduchon manifold $(X,\omega_g)$. A holomorphic pair consists of $E$ and $\phi$ where $\phi\neq 0$ is a holomorphic section of $E$. We denote the Chern connection with respect to a metric $h$ in $E$ by $d_h=\partial_h+\bar{\partial}_E$, and its corresponding curvature by $F_h$.

Given a metric $h_0$ in $E$, any metric $h$ in $E$ can be expressed as $h(s,t)=h_0(fs,t)$ using a positive Hermitian transformation $f$ with respect to $h_0$. We denote this as $h=h_0\cdot f$. Furthermore, we denote the sets of Hermitian transformations and positive Hermitian transformations as $\mathrm{Herm}=\mathrm{Herm}(E,h)$ and $\mathrm{Herm}^+=\mathrm{Herm}^+(E,h)$, respectively.

\begin{definition}[Bradlow \cite{MR1085139}]
  For a holomorphic pair $(E,\phi)$ over a Gauduchon manifold $(X,\omega_g)$, the following equation for a metric $h$ in $E$ is called the vortex equation:
  \begin{gather}\label{vortex-eq}
    i\Lambda_gF_h+\frac{1}{2}\phi\otimes\phi^*_h-\frac{\tau}{2}\mathrm{id}_E=0,~\tau\in \mathbb{R}
  \end{gather}
  where $\phi\otimes\phi^*_h$ is a section of $\mathrm{End}E$, defined by $\phi\otimes\phi^*_h(s)=h(s,\phi)\phi$ for a section $s$ of $E$. A solution $h$ to this equation is called the {\it$\tau$-Hermitian-Einstein metric}.
\end{definition}

The determinant bundle of a reflexive sheaf $\mathcal{F}$ of rank $r$ is defined as $\det\mathcal{F}:=(\wedge^r\mathcal{F})^{**}$. Note that reflexive sheaves of rank $1$ are locally free sheaves (see Okonek-Schneider-Spindler \cite{MR2815674}, Kobayashi \cite{MR909698}). Given a reflexive sheaf $\mathcal{F}$ of rank $r$, we define its degree as:
\begin{gather*}
	\deg\mathcal{F}:=\int_Xc_1(\deg\mathcal{F})\wedge\omega_g^{n-1}
\end{gather*}
The slope of a reflexive sheaf $\mathcal{F}$ is defined as $\mu(\mathcal{F}):=\frac{\deg\mathcal{F}}{\mathrm{rank}\mathcal{F}}$. For a holomorphic pair $(\mathcal{F},\phi)$, we define the following values:
\begin{gather*}
  \begin{aligned}
    &\mu_M:=\sup\left\{~\mu(\mathcal{G})~\left|~
    \begin{aligned}
      &\mathcal{G}\subset\mathcal{F}\text{ is a subsheaf of }0<\mathrm{rank}\mathcal{G}\le\mathrm{rank}\mathcal{F}\\
      &\text{ with quotient torsion-free}  
    \end{aligned}
    \right.~\right\}\\
    &\mu_m(\phi):=\inf\left\{~\mu(\mathcal{F}/\mathcal{G})~\left|~
    \begin{aligned}
      &\mathcal{G}\subset\mathcal{F}\text{ is a subsheaf of }0<\mathrm{rank}\mathcal{G}<\mathrm{rank}\mathcal{F}\\
      &\text{ with quotient torsion-free, and }\phi\in\mathcal{G}
    \end{aligned}
    \right.~\right\}
  \end{aligned}
\end{gather*}

\begin{definition}[Bradlow \cite{MR1085139}]
  A holomorphic pair $(\mathcal{F},\phi)$ is called {\it$\phi$-stable} if it satisfies
  \begin{gather*}
    \mu_M<\mu_m(\phi).
  \end{gather*}
  Moreover, it is called {\it$\tau$-stable} if it satisfies
  \begin{gather*}
    \mu_M<\frac{\tau}{4\pi}\mathrm{Vol}_g(X)<\mu_m(\phi).
  \end{gather*}
\end{definition}

We can prove that the existence of a $\tau$-Hermitian-Einstein metric implies the $\tau$-stability of $(E,\phi)$ in a similar way as Bradlow \cite[Theorem 2.1.6]{MR1085139}.

\begin{theorem}
  If a holomorphic pair $(E,\phi)$ over a compact Gauduchon manifold $(X,\omega_g)$ admits a $\tau$-Hermitian-Einstein metric, then one of the following holds:
  \begin{itemize}
    \item [(a)] $(E,\phi)$ is $\tau$-stable,
    \item [(b)] $E$ decomposes holomorphically as $E=E_\phi\oplus E^{\prime\prime}$, where $\phi\in E_\phi$ is $\tau$-stable, and $E^{\prime\prime}$ is polystable with $\mu(E^{\prime\prime})=\frac{\tau}{4\pi}\mathrm{Vol}_g(X)$.
  \end{itemize}
\end{theorem}

\section{The existence of $\tau$-Hermitian-Einstein Metrics in $\tau$-Stable Bundles}

\subsection{The Solvability of the Perturbed Equation}\label{sol-perturb}
We fix a Hermitian metric $h_0$ in $E$, and consider the perturbed equation for the vortex equation:
\begin{gather}\label{perturbed-equation}
  L_\varepsilon(f):=i\Lambda_g F_{h_0\cdot f}+\frac{1}{2}\phi\otimes\phi^*_{h_0\cdot f}-\frac{\tau}{2}\mathrm{id}_E+\varepsilon\log f=0,~0<\varepsilon\le 1.
\end{gather}
Let $J\subset(0,1]$ be a smooth parameter interval where the perturbed equation $L_\varepsilon(f)=0$ has a smooth solution:
\begin{gather*}
  J:=\{~\varepsilon\in(0,1]~|~^\exists f:(\varepsilon,1]\overset{C^\infty}{\rightarrow}\mathrm{Herm}^+\text{ s.t. }L_{\varepsilon^\prime}(f(\varepsilon^\prime))=0\text{ for all }\varepsilon^\prime\in(\varepsilon,1]~\}
\end{gather*}
In this section, we will show the solvability of the perturbed equation $L_\varepsilon(f)=0$, following the proof of L\"{u}bke-Teleman \cite{MR1370660}, which detailed Uhlenbeck-Yau \cite{MR861491}, Li-Yau \cite{MR915839}.

\bigskip

First, we will choose an initial metric $h_0$ in $E$ such that the equation $L_1(f)=0$ has a solution $f=f_1$, then prove that the interval $J$ is open.

\begin{lemma}
  There exists a metric $h_0$ in $E$ such that the equation $L_1(f)=0$ has a solution $f=f_1$.
\end{lemma}
\begin{proof}
  Take an arbitrary metric $h$ in $E$, and set $K^0_h:=i\Lambda_g F_h+\frac{1}{2}\phi\otimes\phi^*_h-\frac{\tau}{2}\mathrm{id}_E$. Then, setting $h_0:=h\cdot e^{K_h^0}$, we find that $f=e^{-K_h^0}\in\mathrm{Herm}^+(h_0)$ solves the equation $L_1(f)=0$.
\end{proof}

We define $\hat{L}(\varepsilon,f):=f\circ L_\varepsilon(f)$ and denote its derivative in the $f$-direction by $d_2\hat{L}$:
\begin{gather*}
  d_2\hat{L}(\varepsilon,f):L^p_k\mathrm{Herm}\rightarrow L^p_{k-2}\mathrm{Herm}
\end{gather*}
where $L^p_k$ is the Sobolev space with index $k$ in $L^p$-space. For a tangent vector $\varphi$, a straightforward calculation yields
\begin{gather*}
  \begin{aligned}
    d_2\hat{L}(\varepsilon,f)(\varphi)=\varphi\circ L_\varepsilon(f)&-f\circ i\Lambda_g\bar{\partial}(f^{-1}\circ\varphi\circ f^{-1}\circ\partial_0 f-f^{-1}\circ\partial_0\varphi)\\
    &+\frac{1}{2}f\circ\phi\otimes\phi_{h_0}^*\circ\varphi+\varepsilon\varphi
  \end{aligned}
\end{gather*}
where $\partial_0=\partial_{h_0}$.

\begin{proposition}\label{alpha-ineq}
  Let $f\in L^p_k\mathrm{Herm}^+$ and $\varphi\in L^p_k\mathrm{Herm}$ satisfy
  \begin{gather*}
    \hat{L}(\varepsilon,f)=0,~d_2\hat{L}(\varepsilon,f)(\varphi)+\alpha f\circ\log f=0~(\alpha\in\mathbb{R}).
  \end{gather*}
  Then for $\eta:=f^{-1/2}\circ\varphi\circ f^{-1/2}$, we have
  \begin{gather*}
    P(|\eta|^2)+2\varepsilon|\eta|^2+|d^f\eta|^2\le -2\alpha h_0(\log f,\eta)
  \end{gather*}
  where $P:=i\Lambda_g\bar{\partial}\partial$, $|\cdot|:=\sqrt{h_0(\cdot,\cdot)}$, and $d^f\eta$ is defined in the proof.
\end{proposition}
\begin{proof}
  Consider a gauge transformation $\mathrm{Ad}f^{1/2}(\psi):=f^{1/2}\circ\psi\circ f^{-1/2}$ for any smooth section $\psi$ of $\mathrm{End}E$. Define a new $h_0$-connection $d^f$ in $\mathrm{End}E$ by
  \begin{gather*}
    d^f:=\bar{\partial}^f+\partial_0^f,~\bar{\partial}^f:=\mathrm{Ad}f^{\frac{1}{2}}\circ\bar{\partial}\circ\mathrm{Ad}f^{-\frac{1}{2}},~\partial^f_0:=\mathrm{Ad}f^{-\frac{1}{2}}\circ\partial_0\circ\mathrm{Ad}f^{\frac{1}{2}}
  \end{gather*}
  and set $P^f:=i\Lambda_g\bar{\partial}^f\partial_0^f$. From the assumption $\hat{L}(\varepsilon,f)=0$ and simple calculation,
  \begin{gather*}
    \begin{aligned}
      d_2\hat{L}(\varepsilon,f)(\varphi)&=f\circ\left(d_2(i\Lambda_g\bar{\partial}(f^{-1}\circ\partial_0f))(\varphi)\right)+\frac{1}{2}f\circ\phi\otimes\phi_{h_0}^*\circ\varphi+\varepsilon\varphi\\
      &=f^{\frac{1}{2}}\circ P^f(\eta)\circ f^{\frac{1}{2}}+\frac{1}{2}f\circ\phi\otimes\phi_{h_0}^*\circ\varphi+\varepsilon\varphi.
    \end{aligned}
  \end{gather*}
  Further, from the assumption $d_2\hat{L}(\varepsilon,f)(\varphi)+\alpha f\circ\log f=0$,
  \begin{gather*}
    P^f(\eta)+\frac{1}{2}f^{\frac{1}{2}}\circ\phi\otimes\phi_{h_0}^*\circ f^{\frac{1}{2}}\circ\eta+\varepsilon\eta=-\alpha\log f.
  \end{gather*}
  Let $\Phi:=\frac{1}{2}f^{\frac{1}{2}}\circ\phi\otimes\phi_{h_0}^*\circ f^{\frac{1}{2}}$ be a semi-positive Hermitian transformation. Then
  \begin{gather*}
    \begin{aligned}
      &h_0(P^f(\eta),\eta)+h_0(\Phi\circ\eta,\eta)+\varepsilon h_0(\eta,\eta)=-\alpha h_0(\log f,\eta),\\
      &h_0(\eta,P^f(\eta)^*_{h_0})+h_0(\eta,\eta\circ\Phi)+\varepsilon h_0(\eta,\eta)=-\alpha h_0(\eta,\log f).
    \end{aligned}
  \end{gather*}
  Adding both sides, and using $h_0(\Phi\circ\eta,\eta)=h_0(\eta,\eta\circ\Phi)=|\Phi^{1/2}\circ\eta|^2\ge 0$, we obtain
  \begin{gather}\label{ineq-alpha}
    h_0(P^f(\eta),\eta)+h_0(\eta,P^f(\eta)^*_{h_0})+2\varepsilon|\eta|^2\le -2\alpha h_0(\log f,\eta).
  \end{gather}
  Finally, substituting the following equation into the left side yields the desired inequality:
  \begin{gather*}
    P(|\eta|^2)=h_0(P^f(\eta),\eta)+h_0(\eta,P^f(\eta)^*_{h_0})-|d^f\eta|^2.
  \end{gather*}
\end{proof}

\begin{theorem}\label{open-condition}
  $J\subset(0, 1]$ is an open subset, and  there exists $\varepsilon_0\ge 0$ such that $J=(\varepsilon_0,1]$.
\end{theorem}
\begin{proof}
  For $(\varepsilon,f)\in (0,1]\times L^2_k\mathrm{Herm}^+$ satisfying $\hat{L}(\varepsilon,f)=0$, we consider the operator $d_2\hat{L}(\varepsilon,f):L^2_k\mathrm{Herm}{\rightarrow}L^2_{k-2}\mathrm{Herm}$. Since this operator's index is zero, injective and  isomorphism are equivalent. By Proposition \ref{alpha-ineq} ($\alpha=0$) and the Maximal Principle, this operator is injective. Using the Implicit Function Theorem for Banach spaces and the Elliptic Regularity Theorem, we get that the interval $J$ is an open subset in $(0,1]$. From the definition of $J$ and $1\in J$, there exists $\varepsilon_0\ge 0$ such that $J=(\varepsilon_0,1]$.
\end{proof}

\bigskip

In the following arguments, the next theorem is used for norm estimates.

\begin{theorem}[Gilbarg-Trudinger {\cite[Theorem 9.20]{MR1814364}}]\label{max-ineq}
  Let $P:=i\Lambda_g\bar{\partial}\partial$ be an elliptic operator, $f\in C^2(X,\mathbb{R}_{\ge 0})$ be a function, $\lambda\in\mathbb{R}_{\ge 0}$, and $\mu\in\mathbb{R}$. If
  \begin{gather*}
    P(f)\le \lambda f+\mu,
  \end{gather*}
  then there exists a positive constant $C=C(g,\lambda)$ such that
  \begin{gather*}
    \sup_X f\le C(\|f\|_{L^1}+\mu).
  \end{gather*}
\end{theorem}

\begin{proposition}\label{triple-ineq}
  Let us $|\cdot|:=\sqrt{h_0(\cdot,\cdot)},~K^0_{h_0}:=i\Lambda_g F_{h_0}+\frac{1}{2}\phi\otimes\phi^*_{h_0}-\frac{\tau}{2}\mathrm{id}_E$. For the solution $f=f_\varepsilon$ of the perturbed equation $L_\varepsilon(f)=0$, the following inequalities hold:
  \begin{flalign*}
    &(1)~\frac{1}{2}P(|\log f_\varepsilon|^2)+\varepsilon |\log f_\varepsilon|^2\le |K^0_{h_0}||\log f_\varepsilon|,&\\
    &(2)~\sup_X|\log f_\varepsilon|\le \frac{1}{\varepsilon} \sup_X|K^0_{h_0}|,&\\
    &(3)~\sup_X|\log f_\varepsilon|\le C(\|\log f_\varepsilon\|_{L^2}+ \sup_X|K^0_{h_0}|).&
  \end{flalign*}
\end{proposition}
\begin{proof}
  
  (1) We obtain the inequality through calculation by diagonalizing $f$ with a unitary frame (see \cite[Lemma 3.3.4]{MR1370660}).
  \begin{gather}\label{calc-ineq}
    \frac{1}{2}P(|\log f|^2)\le h_0(i\Lambda_g\bar{\partial}(f^{-1}\circ\partial_0f),\log f)
  \end{gather}
  Additionally, since the function $\xi(t):=h_0(\phi\otimes\phi^*_{h_0}\circ e^{t\log f},\log f)$ is monotonically increasing in $t$:
  \begin{gather}\label{monotone}
    h_0(\phi\otimes\phi_{h_0\cdot f}^*-\phi\otimes\phi_{h_0}^*,\log f)=\xi(1)-\xi(0)\ge 0.
  \end{gather}
  Using these relations and the following equation from the perturbed equation yields our claim:
  \begin{gather*}
    \begin{aligned}
      -h_0(K^0_{h_0},\log f)=&~h_0(i\Lambda_g\bar{\partial}(f^{-1}\circ\partial_0f),\log f)\\
      &\quad +\frac{1}{2}h_0(\phi\otimes\phi_{h_0\cdot f}^*-\phi\otimes\phi_{h_0}^*,\log f)+\varepsilon|\log f|^2.
    \end{aligned}
  \end{gather*}

  \noindent
  (2) Let $x_0\in X$ be a maximum point of $|\log f|^2$. If $P(|\log f|^2)(x_0)\ge 0$, then the claim follows together with (1). On the other hand, if we assume $P(|\log f|^2)(x_0)<0$, then by the Maximum Principle, $|\log f|^2$ becomes a constant function around $x_0$, which contradicts $P(|\log f|^2)(x_0)<0$.
  
  \medskip
  \noindent
  (3) From (1), we have $P(|\log f|^2)\le |\log f|^2+\sup_X|K_{h_0}^0|^2$, and using Theorem \ref{max-ineq},
  \begin{gather*}
    \begin{aligned}
      \sup_X|\log f|^2&\le C(\||\log f|^2\|_{L^1}+\sup_X|K_{h_0}^0|^2)\\
      &\le C(\|\log f\|_{L^2}+\sup_X|K_{h_0}^0|)^2.
    \end{aligned}
  \end{gather*}
\end{proof}

\begin{corollary}
  For the solution $f=f_\varepsilon$ of the perturbed equation $L_\varepsilon(f)=0$ and $\varepsilon_0>0$, the following holds:
  \begin{gather*}
     \sup_X|\log f_\varepsilon|\le C~\text{for all }\varepsilon\in(\varepsilon_0,1].
  \end{gather*}
\end{corollary}

\begin{remark}\label{norm-controle}
  Furthermore, one has the following estimates:
  \begin{gather*}
    C_1\le \sup_X|f_\varepsilon|\le C_2~\text{for all }\varepsilon\in(\varepsilon_0,1].
  \end{gather*}
  Consequently, for any smooth section $\psi$ of $\mathrm{End}E$
  \begin{gather*}
    C_3|\psi|\le |f_\varepsilon\circ\psi|\le C_4|\psi|~\text{for all }\varepsilon\in(\varepsilon_0,1].
  \end{gather*}
  We will use these estimates frequently throughout the following arguments without comment.
  \qed
\end{remark}

The next lemma does not require the simplicity for vector bundles.

\begin{lemma}\label{simple-estimate}
  For the solution $f=f_\varepsilon$ of the perturbed equation $L_\varepsilon(f)=0$ and $\varepsilon_0>0$, the following holds:
  \begin{gather*}
     \sup_X\left|\frac{d}{d\varepsilon}f_\varepsilon\right|\le C\text{ for all }\varepsilon\in(\varepsilon_0,1].
  \end{gather*}
\end{lemma}
\begin{proof}
  Let $\varphi=\varphi_\varepsilon=\frac{d}{d\varepsilon}f_\varepsilon\in\mathrm{Herm}$. From $\hat{L}(\varepsilon,f)=f\circ L_varepsilon(f)=0$, we have
  \begin{gather*}
    0=\frac{d}{d\varepsilon}\hat{L}(\varepsilon,f)=d_2\hat{L}(\varepsilon,f)(\varphi)+f\circ\log f.
  \end{gather*}
  By Proposition \ref{alpha-ineq} ($\alpha=1$), we have
  \begin{gather*}
    P(|\eta|^2)+2\varepsilon|\eta|^2\le -2h_0(\log f,\eta)
  \end{gather*}
  where $\eta=f^{-1/2}\circ\varphi\circ f^{-1/2}$. Integrating both sides, and using $\varepsilon_0>0$
  \begin{gather*}
    2\varepsilon\|\eta\|_{L^2}^2\le 2C_1\|\eta\|_{L^2},~\|\eta\|_{L^2}\le C_2.
  \end{gather*}
  Furthermore, from Proposition \ref{alpha-ineq} ($\alpha=1$)
  \begin{gather*}
    P(|\eta|^2)\le -2h_0(\log f,\eta)\le |\eta|^2+C_3.
  \end{gather*}
  Applying Theorem \ref{max-ineq} and the above $L^2$-estimate of $\eta$, we obtain
  \begin{gather*}
     \sup_X|\eta|^2\le C_4.
  \end{gather*}
  The claim follows from this estimate since the norms of $\eta$ and $\varphi$ are equivalent.
\end{proof}

\begin{proposition}\label{L^p_2-estimates}
  For the solution $f=f_\varepsilon$ of the perturbed equation $L_\varepsilon(f)=0$ and $\varepsilon_0>0$, the following hold:
  \begin{flalign*}
    &(1)~\left\|\frac{d}{d\varepsilon}f_\varepsilon\right\|_{L^p_2}\le C_1\left(1+\|f_\varepsilon\|_{L^p_2}\right),&\\
    &(2)~\|f_\varepsilon\|_{L^p_2}\le C_2\left(1+\|f_1\|_{L^p_2}\right).&
  \end{flalign*}
\end{proposition}
\begin{proof}
  (1) Let $P_0=i\Lambda_g\bar{\partial}\partial_0$. Using the Elliptic Regularity Theorem and Lemma \ref{simple-estimate}, we obtain the estimate for $\varphi=\varphi_\varepsilon=\frac{d}{d\varepsilon}f_\varepsilon$
  \begin{gather*}
    \|\varphi\|_{L^p_2}\le C_1(\|P_0\varphi\|_{L^p}+1).
  \end{gather*}
  Differentiating the perturbed equation $L_\varepsilon(f)=0$ with respect to $\varepsilon$
  \begin{gather*}
    \begin{aligned}
      P_0\varphi=&~i\Lambda_g(\bar{\partial}\varphi\circ f^{-1}\circ\partial_0f-\bar{\partial}f\circ f^{-1}\circ\varphi\circ f^{-1}\circ\partial_0f+\bar{\partial}f\circ f^{-1}\circ\partial_0\varphi)\\
      &\quad -\varphi\circ i\Lambda_g F_{h_0}-\frac{1}{2}(\varphi\circ\phi\otimes\phi^*_{h_0}\circ f+f\circ\phi\otimes\phi^*_{h_0}\circ\varphi)\\
      &\quad -f\circ\log f-\varepsilon(\varphi\circ\log f+\varphi).
    \end{aligned}
  \end{gather*}
  Using the Hölder inequality, Lemma \ref{simple-estimate}, and the uniform estimate of $\log f$
  \begin{gather*}
    \|P_0\varphi\|_{L^p}\le C_2(\|\varphi\|_{L^{2p}_1}\|f\|_{L^{2p}_1}+\|f\|_{L^{2p}_1}^2+1).
  \end{gather*}
  According to the above estimates and calculation
  \begin{gather*}
    \begin{aligned}
      \|\varphi\|_{L^p_2}&\le C_1\left(C_2(\|\varphi\|_{L^{2p}_1}\|f\|_{L^{2p}_1}+\|f\|_{L^{2p}_1}^2+1)+1\right)\\
      &\le C_3\left((\|\varphi\|_{L^p_2}^{\frac{1}{2}}+1)(\|f\|_{L^p_2}^{\frac{1}{2}}+1)+(\|f\|_{L^p_2}^{\frac{1}{2}}+1)^2+1\right)\\
      &\le C_3\left(\|\varphi\|_{L^p_2}^{\frac{1}{2}}(\|f\|_{L^p_2}^{\frac{1}{2}}+1)+(\|f\|_{L^p_2}^{\frac{1}{2}}+1+1)^2\right)\\
      &\le C_3\left(\frac{1}{2C_3}\|\varphi\|_{L^p_2}+\frac{C_3}{2}(\|f\|_{L^p_2}^{\frac{1}{2}}+1)^2+(\|f\|_{L^p_2}^{\frac{1}{2}}+2)^2\right).
    \end{aligned}
  \end{gather*}
  In transforming from line 1 to 2, we apply the following Interpolation Theorem:
  \begin{gather*}
    \|\varphi\|_{L^{2p}_1}\le C(\|\varphi\|_{L^p_2}^{\frac{1}{2}}+1),~\|f\|_{L^{2p}_1}\le C^\prime(\|f\|_{L^p_2}^{\frac{1}{2}}+1).
  \end{gather*}
  Transforming the above estimate of $\varphi$
  \begin{gather*}
    \|\varphi\|_{L^p_2}\le C_4(\|f\|_{L^p_2}+\|f\|_{L^p_2}^{\frac{1}{2}}+1)
  \end{gather*}
  and using $\sqrt{x}\le x+1$ for $x\in\mathbb{R}_{\ge 0}$, we obtain (1).

  \medskip
  \noindent
  (2) From $\left|\frac{d}{d\varepsilon}\|f_\varepsilon\|\right|\le \|\frac{d}{d\varepsilon}f_\varepsilon\|$ and (1), we obtain the following inequality:
  \begin{gather*}
    \frac{d}{d\varepsilon}\|f\|_{L^p_2}\ge -\|\varphi\|_{L^p_2}\ge -C(1+\|f\|_{L^p_2}).
  \end{gather*}
  This leads to
  \begin{gather*}
    \log\frac{1+\|f_1\|_{L^p_2}}{1+\|f_\varepsilon\|_{L^p_2}}=\int_{\varepsilon}^{1}\frac{\frac{d}{d\varepsilon^\prime}\|f_{\varepsilon^\prime}\|_{L^p_2}}{1+\|f_{\varepsilon^\prime}\|_{L^p_2}}d\varepsilon^\prime\ge -C(1-\varepsilon)
  \end{gather*}
  which yields
  \begin{gather*}
    \frac{1+\|f_1\|_{L^p_2}}{1+\|f_\varepsilon\|_{L^p_2}}\ge e^{-C(1-\varepsilon)}.
  \end{gather*}
  Thus, (2) follows.
\end{proof}

\begin{theorem}\label{closed-condition}
  $J\subset(0,1]$ is a closed subset, and $J=(0,1]$.
\end{theorem}
\begin{proof}
  Assume that $J=(\varepsilon_0,1]$ with $\varepsilon_0>0$. Then all we need to show is that the solution $f=f_\varepsilon$ of the perturbed equation actually extends to $[\varepsilon_0,1]$, contradicting the openness of $J$. By Proposition \ref{L^p_2-estimates} (2) ($p>2n$) and the Elliptic Regularity Theorem, there exists $\lim_{\varepsilon\rightarrow\varepsilon_0}f_\varepsilon=f_{\varepsilon_0}$, and its limit gives the solution of the equation $L_{\varepsilon_0}(f)=0$. This means that $J=[\varepsilon_0,1]$.
\end{proof}

\subsection{The Construction of a Destabilizing Subsheaf}\label{sol-uniform}
In this section, we will show the uniform $L^2$-estimate of the solution to the perturbed equation \eqref{perturbed-equation}, building on the proofs from \cite{MR944577}, \cite{MR1085139}, \cite{MR4309984}.

By diagonalizing the Hermitian transformation, any $s\in\mathrm{Herm}$ can be expressed by $s=\sum\lambda_ie_i\otimes\theta^i$, where $\lambda_i$ are eigenvalues of $s$.

\begin{definition}\label{fiber-map}
  Given a smooth function $f:\mathbb{R}\rightarrow\mathbb{R}$, we define $f:\mathrm{Herm}\rightarrow\mathrm{End}E$ by
  \begin{gather*}
    f(s)=\sum_{i=1}^{r}f(\lambda_i)e_i\otimes\theta^i.
  \end{gather*}
  We also define $F:\mathrm{Herm}\rightarrow\mathrm{Herm}(\mathrm{End}E)$ from the smooth function $F:\mathbb{R}\times\mathbb{R}\rightarrow\mathbb{R}$ as follows:
  \begin{gather*}
    F(s)(A)=\sum_{i,j=1}^{r}F(\lambda_j,\lambda_i)A^i_je_i\otimes\theta^j
  \end{gather*}
  where $A=\sum A^i_je_i\otimes\theta^j\in\mathrm{End}E$. Furthermore, let $df:\mathbb{R}\times\mathbb{R}\rightarrow\mathbb{R}$ be defined by
  \begin{gather*}
    df(x,y)=
    \left\{
      \begin{aligned}
        \frac{f(x)-f(y)}{x-y}&\text{ if }x\neq y\\
        \frac{df}{dx}~~~~~~&\text{ if }x=y.
      \end{aligned}
    \right.
  \end{gather*}\qed
\end{definition}

We define a smooth function $\Psi:\mathbb{R}\times\mathbb{R}\rightarrow\mathbb{R}$ by
  \begin{gather*}
    \Psi(x,y)=
    \left\{
      \begin{aligned}
        \frac{e^{y-x}-1}{y-x}&\text{ if }x\neq y\\
        1~~~~~&\text{ if }x=y.
      \end{aligned}
    \right.
  \end{gather*}

\begin{lemma}[Nie-X.Zhang {\cite[p.635]{MR3745874}}] For $f\in\mathrm{Herm}^+,~s=\log f\in\mathrm{Herm}$, the following holds:
  \begin{gather*}
    i\Lambda_g\mathrm{tr}((f^{-1}\circ\partial_0f)\circ\bar{\partial}s)=\left<\Psi(s)(\bar{\partial}s),\bar{\partial}s\right>_{h_0,g}
  \end{gather*}
  where the inner product $\left<\cdot,\cdot\right>_{h_0,g}$ on the right-hand side is induced from $h_0,g$.
\end{lemma}

Taking the inner product of the pertured equation \eqref{perturbed-equation} with $s=\log f$, integrating, and using the above lemma, we obtain
\begin{gather}\label{eq-1}
  -\varepsilon\|s\|_{L^2}^2=\int_X\left<K^0,s\right>_{h_0}+\int_X\left<\Psi(s)(\bar{\partial}s),\bar{\partial}s\right>_{h_0,g}
\end{gather}
where $K^0:=i\Lambda_g F_{h_0}-\frac{\tau}{2}\mathrm{id}_E+\frac{1}{2}\phi\otimes\phi^*_{h_0}$.

\begin{theorem}\label{L^2-uniform}
  Let $f=f_\varepsilon$ be the solution of the perturbed equation $L_\varepsilon(f)=0$ on a holomorphic pair $(E,\phi)$. If the holomorphic pair is $\tau$-stable, then there exists a positive constant $C$ independent of $\varepsilon$ such that
  \begin{gather*}
    \|\log f_\varepsilon\|_{L^2}\le C\text{ for all }\varepsilon\in(0,1].
  \end{gather*}
\end{theorem}
\begin{proof}
  We show the contraposition. Let $\limsup_{\varepsilon\rightarrow 0}\|\log f_{\varepsilon}\|_{L^2}=\infty$, and take a subsequence $\varepsilon_i\rightarrow 0$ as $i\rightarrow\infty$ such that $\lim_{i\rightarrow \infty}\|\log f_{\varepsilon_i}\|_{L^2}=\infty$. We set
  \begin{gather*}
    f_i:=f_{\varepsilon_i},~s_i:=\log f_i,~l_i:=\|\log f_i\|_{L^2},~u_i:=\frac{s_i}{l_i}.
  \end{gather*}
  In particular, $\sup_X|u_i|\le C$ from Theorem \ref{triple-ineq} (3). Also, from \eqref{eq-1}
  \begin{gather*}
    \begin{aligned}
      -\varepsilon_il_i+\frac{1}{2el_i}\|\phi\|^2_{L^2}=&\int_X\left<i\Lambda_g F_{h_0}-\frac{\tau}{2}\mathrm{id}_E,u_i\right>
      +\int_X \left<l_i\Psi(l_iu_i)(\bar{\partial}u_i),\bar{\partial}u_i\right>\\
      &\quad +\int_X\frac{1}{2}\left(\left<\phi\otimes\phi^*_{h_0}\circ f_i,u_i\right>+\frac{1}{el_i}|\phi|^2\right).
    \end{aligned}
  \end{gather*}
  Let $\xi:\mathbb{R}\times\mathbb{R}\rightarrow\mathbb{R}_{\ge 0}$ be an arbitrary smooth non-negative function satisfying $\xi(x,y)<(x-y)^{-1}$ when $x>y$. From the behaviour of $\Psi$
  \begin{gather*}
    l_i\Psi(l_ix,l_iy)\rightarrow
    \left\{
    \begin{aligned}
      \frac{1}{x-y}&\text{ for }x>y\\
      \infty~~&\text{ for }x\le y,
    \end{aligned}
    \right.
  \end{gather*}
  we can assume that
  \begin{gather*}
    \xi(x,y)<l_i\Psi(l_ix,l_iy)
  \end{gather*}
  for sufficiently large $i$. Let $\eta:\mathbb{R}\rightarrow\mathbb{R}_{\ge 0}$ be an arbitrary smooth non-negative function satisfying $\eta(x)=0$ when $x\le \delta$, where $\delta$ is a sufficiently small positive constant. Then for sufficiently large $i$, we have
  \begin{gather*}
    \begin{aligned}
      \eta(x)&\le\frac{1}{2}\left(e^{l_ix}x+\frac{1}{el_i}\right),\\
      \left<\eta(u_i)(\phi),\phi\right>&\le\frac{1}{2}\left(\left<\phi\otimes\phi^*_{h_0}\circ f_i,u_i\right>+\frac{1}{el_i}|\phi|^2\right).
    \end{aligned}
  \end{gather*}
  Hence we obtain the following inequality.
  \begin{gather}\label{ineq-1}
    \begin{aligned}
      \int_X\left<i\Lambda_g F_{h_0}-\frac{\tau}{2}\mathrm{id}_E,u_i\right>&+\int_X\left<\xi(u_i)(\bar{\partial}u_i),\bar{\partial}u_i\right>\\
      &+\int_X\left<\eta(u_i)(\phi),\phi\right>\le -\varepsilon_il_i+\frac{1}{2el_i}\|\phi\|^2_{L^2}
    \end{aligned}
  \end{gather}
  Since $\sup_X|u_i|\le C$, there exists a positive constant $c$ such that \\
  $\left<\xi(u_i)(\bar{\partial}u_i),\bar{\partial}u_i\right>\ge c\left<\bar{\partial}u_i,\bar{\partial}u_i\right>$. Then, if $\eta=0$, there exists $u_\infty\in L^2_1$ such that
  \begin{gather*}
    u_i\rightharpoonup u_\infty\text{ in }L^2_1.
  \end{gather*}
  Moreover, $u_\infty\neq 0$ since $\|u_i\|_{L^2}=1$. Taking the limit as $i\rightarrow\infty$ in inequality \eqref{ineq-1}, we obtain (see \cite[Lemma 5.4]{MR944577})
  \begin{equation}\label{int-ineq}
    \begin{aligned}
      \int_X\left<i\Lambda_g F_{h_0}-\frac{\tau}{2}\mathrm{id}_E,u_\infty\right>_{h_0}&+\int_X\left<\xi(u_\infty)(\bar{\partial}u_\infty),\bar{\partial}u_\infty\right>_{h_0,g}\\
      &+\int_X\left<\eta(u_\infty)(\phi),\phi\right>_{h_0}\le 0.
    \end{aligned}
  \end{equation}

  The eigenvalues of $u_\infty$ are constant almost everywhere and have a non-positive eigenvalue from \cite[Lemma 3.9.2, Proposition 3.10.2]{MR1085139}. We consider two cases for the eigenvalues of $u_\infty$.

  \noindent
  \textbf{Case 1.}
  Suppose all eigenvalues of $u_\infty$ are equal. Then we can write $u_\infty=\lambda\cdot\mathrm{id}_E< 0$. For $\nu:=\lambda\cdot\mathrm{rank}E\left(\mu(E)-\frac{\tau}{4\pi}\mathrm{Vol}_g(X)\right)$, we have
  \begin{gather*}
      2\pi\nu=\lambda\cdot\mathrm{rank}E\left(2\pi\deg E-\frac{\tau}{2}\mathrm{Vol}_g(X)\right)=\int_X\mathrm{tr}\left((i\Lambda_g F_{h_0}-\frac{\tau}{2}\mathrm{id}_E)\circ u_\infty\right).
  \end{gather*}
  From the \eqref{int-ineq} with $\xi=\eta=0$, then $\nu\le 0$. Since $\lambda<0$,
  \begin{gather*}
    \mu(E)-\frac{\tau}{4\pi}\mathrm{Vol}_g(X)\ge 0.
  \end{gather*}
  This implies the $\tau$-unstability of the holomorphic pair $(E,\phi)$.

  \noindent
  \textbf{Case 2.}
  Suppose $u_\infty$ has different eigenvalues. Let the eigenvalues of $u_\infty$ be $\lambda_1<\lambda_2<\cdots<\lambda_l~(2\le l\le r)$ and $P_i:\mathbb{R}\rightarrow\mathbb{R}~(1\le i\le l-1)$ be a smooth function satisfying
  \begin{gather*}
    P_i(x)=1\text{ on }x\le \lambda_i,~P_i(x)=0\text{ on }x\ge \lambda_{i+1}.
  \end{gather*}
  We set $\pi_i:=P_i(u_\infty)$, and according to \cite[Proposition 3.10.2]{MR1085139}, $\pi_i$ are weakly holomorphic subsheaves ($\pi_i\in L^2_1\mathrm{End}E$, $\pi_i^2=\pi_i=\pi_i^*$ and $(\mathrm{id}_E-\pi_i)\circ\bar{\partial}\pi_i=0$ in $L^1$). By Uhlenbeck-Yau's regularity statement of $L^2_1$ subsheaf \cite{MR861491}, \cite{MR2180378}, $\pi_i$ determines subsheaves of $E$. These subsheaves are denoted by $E_i=\pi_i(E)$. We define
  \begin{gather*}
    \begin{aligned}
      \nu:=&\lambda_l\cdot\mathrm{rank}E\left(\mu(E)-\frac{\tau}{4\pi}\mathrm{Vol}_g(X)\right)\\
      &\quad -\sum_{i=1}^{l-1}(\lambda_{i+1}-\lambda_i)\cdot\mathrm{rank}E_i\left(\mu(E_i)-\frac{\tau}{4\pi}\mathrm{Vol}_g(X)\right).
    \end{aligned}
  \end{gather*}
  From \cite[Lemma 3.12.1]{MR1085139},
  \begin{gather*}
    \nu\le 0.
  \end{gather*}
  This implies the $\tau$-unstability of the holomorphic pair $(E,\phi)$ (see \cite[Lemma 3.13.1]{MR1085139}).
\end{proof}

\begin{corollary}
  For the solution $f=f_\varepsilon$ of the perturbed equation $L_\varepsilon(f)=0$, there exists a positive constant $C$ such that
  \begin{gather*}
     \sup_X|\log f_\varepsilon|\le C~\text{for all }\varepsilon\in(0,1].
  \end{gather*}  
\end{corollary}

\subsection{The Solvability of the Vortex Equation}\label{sol-vol}
In this section, we will show the solvability of the vortex equation \eqref{vortex-eq}.

\begin{lemma}
  For a solution $f=f_\varepsilon$ of the perturbed equation $L_\varepsilon(f)=0$, there exists a $C^0$-convergence $f_0$ such that
  \begin{gather*}
    f_\varepsilon\rightarrow f_0\text{ in }C^0      
  \end{gather*}
  where $f_0$ is a continuous positive Hermitian transformation.
\end{lemma}
\begin{proof}
  The perturbed equation $L_\varepsilon(f)=0$ implies
  \begin{gather*}
    \begin{aligned}
      &i\Lambda_g(\bar{\partial}f\circ f^{-1}\circ\partial_0 f)\\
      &\quad =i\Lambda_g\bar{\partial}\partial_0f+f\circ(i\Lambda_g F_{h_0}+\frac{1}{2}\phi\otimes\phi^*_{h_0}\circ f-\frac{\tau}{2}\mathrm{id}_E+\varepsilon\log f).
    \end{aligned}
  \end{gather*}
  By integrating, we have
  \begin{gather*}
    \|f_\varepsilon\|_{L^2_1}\le C\text{ for all }\varepsilon\in (0,1].
  \end{gather*}
  Therefore, there exists an $L^2$ convergent subsequence $\{f_{\varepsilon_i}\}$ such that $f_{\varepsilon_i}\rightarrow f_{\varepsilon_\infty}$. Let us write
  \begin{gather*}
    f_{\varepsilon_i}=:f_i,~h_0\cdot f_i=:h_i,~f_i^{-1}\circ f_j=:f_{ij}\in\mathrm{Herm}^+(h_i),
  \end{gather*}
  and the $L^2$-convergence of $\log f_{ij}$ follows:
  \begin{gather*}
    \|\log f_{ij}\|_{L^2}^2=\int_{X}|\log f_{ij}|_{h_0}^2\le C(\|f_{ij}-\mathrm{id}_E\|_{L^2}^2+\|f_{ji}-\mathrm{id}_E\|_{L^2}^2)\rightarrow 0,
  \end{gather*}
  where we have used the inequality $x^2\le(e^x-1)^2+(e^{-x}-1)^2$ for $x\in\mathbb{R}_{\ge 0}$ to estimate the right-hand side.

  The perturbed equation $L_\varepsilon(f)=0$ for $h_0$ can be rewritten as an equation for $h_i$:
  \begin{gather*}
    \begin{aligned}
      0&=i\Lambda_gF_{h_0\cdot f_j}+\frac{1}{2}\phi\otimes\phi_{h_0\cdot f_j}^*-\frac{\tau}{2}\mathrm{id}_E+\varepsilon_j\log f_j\\
      &=i\Lambda_gF_{h_i}+i\Lambda_g\bar{\partial}(f_{ij}^{-1}\circ\partial_i f_{ij})+\frac{1}{2}\phi\otimes\phi_{h_i\cdot f_{ij}}^*-\frac{\tau}{2}\mathrm{id}_E+\varepsilon_j\log f_j\\
      &=-\varepsilon_i\log f_i+i\Lambda_g\bar{\partial}(f_{ij}^{-1}\circ\partial_i f_{ij})+\frac{1}{2}(\phi\otimes\phi_{h_i\cdot f_{ij}}^*-\phi\otimes\phi_{h_i}^*)+\varepsilon_j\log f_j
    \end{aligned}
  \end{gather*}
  where $\partial_i=\partial_{h_i}$. When in transforming the equation from line 2 to 3, we used the equation $L_{\varepsilon_i}(f_i)=0$. Form the above equation and the inequality \eqref{monotone},
  \begin{gather*}
    h_i(i\Lambda_g\bar{\partial}(f_{ij}^{-1}\circ\partial_i f_{ij}),\log f_{ij})\le h_i(\varepsilon_i\log f_i-\varepsilon_j\log f_j,\log f_{ij}).
  \end{gather*}
  Then from the inequality \eqref{calc-ineq},
  \begin{gather*}
    \begin{aligned}
      P(|\log f_{ij}|_{h_i}^2)&\le 2h_i(i\Lambda_g\bar{\partial}(f_{ij}^{-1}\circ\partial_if_{ij}),\log f_{ij})\\
      &\le 2h_i(\varepsilon_i\log f_i-\varepsilon_j\log f_j,\log f_{ij})\\
      &\le C(\varepsilon_i^2+\varepsilon_j^2+|\log f_{ij}|_{h_i}^2).
    \end{aligned}
  \end{gather*}
  From Theorem \ref{max-ineq} and since $h_0,h_i$ define an equivalent norm,
  \begin{gather*}
     \sup_X|\log f_{ij}|_{h_0}^2\le C\left(\|\log f_{ij}\|_{L^2}^2+(\varepsilon_i+\varepsilon_j)^2\right)\rightarrow 0.
  \end{gather*}
  This implies the existence of a $C^0$-convergence $f_0$.
\end{proof}

The next lemma provides an $L^p_2$-estimate of the solution to the perturbed equation $L_\varepsilon(f)=0$ without requiring the simplicity of vector bundles.

\begin{lemma}[Donaldson {\cite[Lemma 19]{MR765366}}]\label{Donaldson}
  Given the family $\{f_\varepsilon\}_{\varepsilon\in(0,1]}$ of positive Hermitian transformations with respect to $h_0$ satisfying the following two conditions:
  \begin{itemize}
    \item [(1)] There exists a continuous positive Hermitian transformation $f_0$ such that $f_\varepsilon\rightarrow f_0$ in $C^0$,
    \item [(2)] $\sup_X|i\Lambda_g F_{h_0\cdot f_\varepsilon}|\le C_1$ for all $\varepsilon\in(0,1]$.
  \end{itemize}
  Then, the family has a uniform $C^1$-estimate
  \begin{gather*}
     \sup_X|d_0f_\varepsilon|\le C_2\text{ for all }\varepsilon\in(0,1].
  \end{gather*}
  Furthermore, the following uniform $L^p_2$-estimate holds:
  \begin{gather*}
    \|f_\varepsilon\|_{L^p_2}\le C_3\text{ for all }\varepsilon\in(0,1].
  \end{gather*}\qed
\end{lemma}

\begin{theorem}
  For the solution $f=f_\varepsilon$ of the perturbed equation $L_\varepsilon(f)=0$ on a $\tau$-stable holomorphic pair $(E,\phi)\rightarrow X$,
  \begin{gather*}
    f_0:=\lim_{\varepsilon\rightarrow 0} f_\varepsilon
  \end{gather*}
  gives a solution for $L_0(f)=0$. In other words, there exists a $\tau$-Hermitian-Einstein metric on the $\tau$-stable holomorphic pair.
\end{theorem}
\begin{proof}
  By Lemma \ref{Donaldson} ($p>2n$) and the Elliptic Regularity Theorem, there exists $\lim_{\varepsilon\rightarrow 0}f_\varepsilon=:f_0$, and its limit gives the solution of the equation $L_0(f)=0$.
\end{proof}

\subsection{The Uniqueness of $\tau$-Hermitian-Einstein Metrics}\label{sol-unique}
In this section, we will show the uniqueness of $\tau$-Hermitian-Einstein metric, using the $\phi$-simplicity of the holomorphic pair $(E,\phi)$:

\begin{definition}
  A holomorphic pair $(E,\phi)$ with $\phi\neq 0$ is called {\it$\phi$-simple} if it satisfies the following condition:
  \begin{gather*}
    u\in H^0(\mathrm{End}E),~u(\phi)=0~\Rightarrow~u=0.
  \end{gather*}
\end{definition}

\begin{proposition}
  If a holomorphic pair $(E,\phi)$ is $\phi$-stable, then it is $\phi$-simple.
\end{proposition}
\begin{proof}
  We prove by contraposition. If we take $0\neq u\in H^0(\mathrm{End}E)$ with $u(\phi)=0$, then $\mu_m(\phi)\le\mu(E/\mathrm{Ker}u)=\mu(\mathrm{Im}u)\le\mu_M$. This implies the $\phi$-unstability of the holomorphic pair $(E,\phi).$
\end{proof}

\begin{proposition}
  The $\tau$-Hermitian-Einstein metrics is unique if $(E,\phi)$ is $\phi$-simple.
\end{proposition}
\begin{proof}
  Let take two $\tau$-Hermitian-Einstein metric $h,~k$ and let $h=k\cdot f,~f=g^2$ for some $f,g\in\mathrm{Herm}^+(k)$. Take the Chern connection $d_k=\partial_k+\bar{\partial}$ for $(k,\bar{\partial})$, and define the new $k$-connection $d^\prime$ by
  \begin{gather*}
      d^\prime:=\partial^\prime_k+\bar{\partial}^\prime,~\partial^\prime_k:=g^{-1}\circ\partial_k\circ g,~\bar{\partial}^\prime:=g\circ\bar{\partial}\circ g^{-1}
  \end{gather*}
  In particular, $d^\prime$ is the Chern connection of $(\bar{\partial}^\prime,k)$. Let $\mathcal{E}=(E,\bar{\partial}),~\mathcal{E}^\prime=(E,\bar{\partial}^\prime)$, then $g$ is a holomorphic section of $\mathrm{Hom}(\mathcal{E},\mathcal{E}^\prime)$. We set $K_E:=i\Lambda_g F_k$, the mean curvature of the $\tau$-Hermitian-Einstein metric $k$
  \begin{gather*}
    \begin{aligned}
      K_{\mathrm{End}}&=K_E\otimes\mathrm{id}_{E^*}+\mathrm{id}_E\otimes K_{E^*}\\
      &=\left(-\frac{1}{2}\phi\otimes\phi^*_k+\frac{\tau}{2}\mathrm{id}_E\right)\otimes\mathrm{id}_{E^*}+\mathrm{id}_E\otimes\left(\frac{1}{2}\phi_k^*\otimes\phi-\frac{\tau}{2}\mathrm{id}_{E^*}\right)\\
      &=-\frac{1}{2}((\phi\otimes\phi^*_k)\otimes\mathrm{id}_{E^*}+\mathrm{id}_E\otimes(\phi_k^*\otimes\phi))\\
      &\le 0\\
    \end{aligned}
  \end{gather*}
  Therefore, applying the Bochner's Vanishing Theorem, we obtain $d_{\mathrm{Hom}}g=0$. Hence we have
  \begin{gather*}
    0=\partial_\mathrm{Hom}g=\partial^\prime_k\circ g-g\circ\partial_k=g^{-1}\circ\partial_kg^2.
  \end{gather*}
  Since $g^2=f$, it follows that $\bar{\partial}(\log f)=0$. Compute the Hermitian transformation by diagonalizing $\log f=\sum\lambda_ie_i\otimes\theta^i$.
  \begin{gather*}
    \begin{aligned}
      K_{\mathrm{End}}(\log f)=-\frac{1}{2}\sum_{i=1}^{r}\lambda_i\bar{\phi}^i(\phi\otimes\theta^i+e_i\otimes\phi^*)
    \end{aligned}
  \end{gather*}
  where $\phi=\sum\phi^ie_i$. And then,
  \begin{gather*}
    \begin{aligned}
      \left<K_{\mathrm{End}}(\log f),\log f\right>_k=-\sum_{i=1}^{r}|\lambda_i|^2|\phi^i|^2.
    \end{aligned}
  \end{gather*}
  Whereas,
  \begin{gather*}
      \log f(\phi)=\sum_{i=1}^{r}\lambda_i\phi^ie_i.
  \end{gather*}
  The left side of the resulting equality $\left<K_{\mathrm{End}}(\log f),\log f\right>_k=-|\log f(\phi)|^2$ is $0$ from the Bochner's Vanishing Theorem. Thus, we get $\log f(\phi)=0$. By the $\phi$-simplicity, we conclude that $\log f=0$, and therefore $f=\mathrm{id}_E$ and $h=k$.
\end{proof}

\section{Higgs Bundles Version}
The existence of a canonical metric on stable Higgs bundles over a compact Gauduchon manifold has been proved by Jacob \cite{MR3372463} using the heat flow method. In this section, we give another proof of his theorem, using the same approach as in the proof of Theorem \ref{main-theorem}.

\begin{theorem}
  If a Higgs bundle $(E,\theta)$ over an $n$-dimensional compact Gauduchon manifold $(X,\omega_g)$ is stable, then there exists a Hermitian-Einstein metric, and it is unique up to a positive scalar.
\end{theorem}

Taking an appropriate initial metric $h_0$ on $E$, we consider the perturbed equation:
\begin{gather*}
  L_\varepsilon(f)=i\Lambda_g(F_h+[\theta,\theta^*_h])-\lambda\mathrm{id}_E+\varepsilon\log f=0,~0<\varepsilon\le 1.
\end{gather*}
The proof for the solvability is the same as in Section \ref{sol-perturb}. In particular, it suffices to show that Proposition \ref{alpha-ineq} and the inequality \eqref{monotone} also hold for the case of Higgs bundles.

\begin{proposition}
  Let $f\in L^p_k\mathrm{Herm}^+$ and $\varphi\in L^p_k\mathrm{Herm}$ satisfy
  \begin{gather*}
    \hat{L}(\varepsilon,f)=0,~d_2\hat{L}(\varepsilon,f)(\varphi)+\alpha f\circ\log f=0~(\alpha\in\mathbb{R}).
  \end{gather*}
  Then for $\eta:=f^{-1/2}\circ\varphi\circ f^{-1/2}$, we have
  \begin{gather*}
    P(|\eta|^2)+2\varepsilon|\eta|^2+|d^f\eta|^2\le -2\alpha h_0(\log f,\eta).
  \end{gather*}
\end{proposition}
\begin{proof}
  We use the same notations as in Proposition \ref{alpha-ineq}.
  From the assumption $\hat{L}(\varepsilon,f)=0$ and direct calculation,
  \begin{gather*}
    \begin{aligned}
      d_2\hat{L}(\varepsilon,f)(\varphi)=&f^{\frac{1}{2}}\circ P^f(\eta)\circ f^{\frac{1}{2}}\\
      &\quad +f\circ i\Lambda_g(-[\theta,f^{-1}\circ\varphi\circ f^{-1}\theta^*_{h_0}\circ f]+[\theta,f^{-1}\circ\theta^*_{h_0}\circ\varphi])\\
      &\quad +\varepsilon\varphi.
    \end{aligned}
  \end{gather*}
  Further, from the assumption $d_2\hat{L}(\varepsilon,f)(\varphi)+\alpha f\circ\log f=0$,
  \begin{gather*}
    P^f(\eta)+\Theta+\varepsilon\eta=-\alpha\log f
  \end{gather*}
  where $\Theta:=f^{1/2}\circ i\Lambda_g(-[\theta,f^{-1}\circ\varphi\circ f^{-1}\theta^*_{h_0}\circ f]+[\theta,f^{-1}\circ\theta^*_{h_0}\circ\varphi])\circ f^{-1/2}$. Then
  \begin{gather*}
    \begin{aligned}
      &h_0(P^f(\eta),\eta)+h_0(\Theta,\eta)+\varepsilon h_0(\eta,\eta)=-\alpha h_0(\log f,\eta),\\
      &h_0(\eta,P^f(\eta)^*_{h_0})+h_0(\eta,\Theta)+\varepsilon h_0(\eta,\eta)=-\alpha h_0(\eta,\log f).
    \end{aligned}
  \end{gather*}
  Adding both sides and using $h_0(\Theta,\eta)=h_0(\eta,\Theta)=|[f^{1/2}\circ\theta\circ f^{-1/2},\eta]|_{h_0,g}\ge 0$, we obtain
  \begin{gather*}
    h_0(P^f(\eta),\eta)+h_0(\eta,P^f(\eta)^*_{h_0})+2\varepsilon|\eta|^2\le -2\alpha h_0(\log f,\eta).
  \end{gather*}
  The rest of the proof is similar to Proposition \ref{alpha-ineq}.
\end{proof}

\begin{lemma}
  The function $\xi(t):=h_0(i\Lambda_g[\theta,e^{-ts}\circ\theta^*_{h_0}\circ e^{ts},s])$ where $s=\log f$ is monotonically increasing in $t$. 
\end{lemma}
\begin{proof}
  This follows from a straightforward calculation
  \begin{gather*}
    \begin{aligned}
      \xi^\prime(t)&=i\Lambda_g\left<[\theta,-s\circ e^{-ts}\circ\theta^*_{h_0}\circ e^{ts}+e^{-ts}\circ\theta^*_{h_0}\circ s\circ e^{ts}],s\right>_{h_0}\\
      &=\left|[s,e^{ts/2}\circ\theta\circ e^{-ts/2}]\right|_{h_0,g}^2\ge 0.
    \end{aligned}
  \end{gather*}
\end{proof}

We then consider the uniform estimates of the solution to the perturbed equation. This follows the discussion in Section \ref{sol-uniform} and is similar to C.Zhang-P.Zhang-X.Zhang \cite{MR4237961}. We therefore omit here.

The solvability of the curvature equation follows a method analogous to Section \ref{sol-vol} and is also omitted.

Lastly, we establish the uniqueness of the Hermitian-Einstein metric by the simplicity of the Higgs bundle. The next lemma is obtained in the same way as for holomorphic vector bundles.

\begin{lemma}
  If a Higgs bundle $(E,\theta)$ is stable, then it is simple:
  \begin{gather*}
    u\in\Gamma(\mathrm{End}E),~\bar{\partial}_Eu=[\theta,u]=0~\Rightarrow~u\in\mathbb{C}\cdot\mathrm{id}_E.
  \end{gather*}
\end{lemma}

The following proposition is analogous to the Bochner's Vanishing Theorem, and its proof is omitted.

\begin{proposition}\label{Bochner}
  Let $(E,\theta)$ be a Higgs bundle over a compact Gauduchon manifold $(X,\omega_g)$ with a metric $h$ in $E$, and let $K_h:=i\Lambda_g(F_h+[\theta,\theta^*_h])$. Furthermore, let $D^{\prime\prime}:=\bar{\partial}+\theta,~D^\prime:=\partial_h+\theta^*_h$. If $h\cdot K_h$ is semi-negative definite at each point of $X$, then the following holds:
  \begin{gather*}
    D^{\prime\prime}s=0~\Rightarrow~D^{\prime}s=0 \text{ for }s\in\Gamma(\mathrm{End}E).
  \end{gather*}
\end{proposition}

\begin{proposition}
  The Hermitian-Einstein metric is unique if $(E,\theta)$ is simple.
\end{proposition}
\begin{proof}
  Let us take two Hermitian-Einstein metrics $h,~k$ and let $h=k\cdot f,~f=g^2$ for some $f,g\in\mathrm{Herm}^+(k)$. Furthermore, let $D^{\prime\prime}:=\bar{\partial}+\theta,~D^\prime_k:=\partial_k+\theta^*_k$. Take the connection $D_k=D^\prime_k+D^{\prime\prime}$, and define the new $k$-connection $\tilde{D}_k$ by
  \begin{gather*}
      \tilde{D}_k:=\tilde{D}^\prime_k+\tilde{D}^{\prime\prime},~\tilde{D}^\prime_k:=g^{-1}\circ D^\prime_k\circ g,~\tilde{D}^{\prime\prime}:=g\circ D^{\prime\prime}\circ g^{-1}
  \end{gather*}
  In particular, $g$ is a section of $\mathrm{End}E$ satisfying
  \begin{gather*}
    D^{\prime\prime}_{\mathrm{End}}g=\tilde{D}^{\prime\prime}\circ g-g\circ D^{\prime\prime}=0.
  \end{gather*}
  We set $K_E:=i\Lambda_g(F_k+[\theta,\theta^*_k])$,
  \begin{gather*}
    K_{\mathrm{End}}=K_E\otimes\mathrm{id}_{E^*}+\mathrm{id}_E\otimes K_{E^*}=0.
  \end{gather*}
  Therefore, applying Theorem \ref{Bochner} yields $D^\prime_{\mathrm{End}}g=0$, hence we have
  \begin{gather*}
    0=D^\prime_\mathrm{End}g=\tilde{D}^\prime_k\circ g-g\circ D^\prime_k=g^{-1}\circ D^\prime_kg^2.
  \end{gather*}
  Since $g^2=f$, it follows that $D^{\prime\prime}f=0$. By the simplicity of the Higgs bundle, we conclude that $f\in\mathbb{C}\cdot\mathrm{id}_E$, and therefore $h=ak$ for some $a\in\mathbb{R}_{>0}$.
\end{proof}

\section*{Acknowledgement}
I am deeply grateful to my academic advisor, Professor Nobuhiro Honda, and my laboratory colleagues, Ryota Kotani and Jun Sasaki, for their invaluable guidance on both my research and academic writing.

%The author is grateful to Nobuhiro Honda, Ryota Kotani and Jun Sasaki, who provided valuable feedback on the presentation of this work.

\bibliographystyle{abbrv}
\bibliography{bibtex}

\begin{thebibliography}{10}

\bibitem{MR1086749}
S.~B. Bradlow.
\newblock Vortices in holomorphic line bundles over closed {K}\"ahler
  manifolds.
\newblock {\em Comm. Math. Phys.}, 135(1):1--17, 1990.

\bibitem{MR1085139}
S.~B. Bradlow.
\newblock Special metrics and stability for holomorphic bundles with global
  sections.
\newblock {\em J. Differential Geom.}, 33(1):169--213, 1991.

\bibitem{MR710055}
S.~K. Donaldson.
\newblock A new proof of a theorem of {N}arasimhan and {S}eshadri.
\newblock {\em J. Differential Geom.}, 18(2):269--277, 1983.

\bibitem{MR765366}
S.~K. Donaldson.
\newblock Anti self-dual {Y}ang-{M}ills connections over complex algebraic
  surfaces and stable vector bundles.
\newblock {\em Proc. London Math. Soc. (3)}, 50(1):1--26, 1985.

\bibitem{MR885784}
S.~K. Donaldson.
\newblock Infinite determinants, stable bundles and curvature.
\newblock {\em Duke Math. J.}, 54(1):231--247, 1987.

\bibitem{MR1814364}
D.~Gilbarg and N.~S. Trudinger.
\newblock {\em Elliptic partial differential equations of second order}.
\newblock Classics in Mathematics. Springer-Verlag, Berlin, 2001.
\newblock Reprint of the 1998 edition.

\bibitem{MR3372463}
A.~Jacob.
\newblock Stable {H}iggs bundles and {H}ermitian-{E}instein metrics on
  non-{K}\"ahler manifolds.
\newblock In {\em Analysis, complex geometry, and mathematical physics: in
  honor of {D}uong {H}. {P}hong}, volume 644 of {\em Contemp. Math.}, pages
  117--140. Amer. Math. Soc., Providence, RI, 2015.

\bibitem{MR556302}
S.~Kobayashi.
\newblock First {C}hern class and holomorphic tensor fields.
\newblock {\em Nagoya Math. J.}, 77:5--11, 1980.

\bibitem{MR664562}
S.~Kobayashi.
\newblock Curvature and stability of vector bundles.
\newblock {\em Proc. Japan Acad. Ser. A Math. Sci.}, 58(4):158--162, 1982.

\bibitem{MR909698}
S.~Kobayashi.
\newblock {\em Differential geometry of complex vector bundles}, volume~15 of
  {\em Publications of the Mathematical Society of Japan}.
\newblock Princeton University Press, Princeton, NJ; Princeton University
  Press, Princeton, NJ, 1987.
\newblock Kan\^o{} Memorial Lectures, 5.

\bibitem{MR915839}
J.~Li and S.-T. Yau.
\newblock Hermitian-{Y}ang-{M}ills connection on non-{K}\"ahler manifolds.
\newblock In {\em Mathematical aspects of string theory ({S}an {D}iego,
  {C}alif., 1986)}, volume~1 of {\em Adv. Ser. Math. Phys.}, pages 560--573.
  World Sci. Publishing, Singapore, 1987.

\bibitem{MR701206}
M.~L\"ubke.
\newblock Stability of {E}instein-{H}ermitian vector bundles.
\newblock {\em Manuscripta Math.}, 42(2-3):245--257, 1983.

\bibitem{MR1370660}
M.~L\"ubke and A.~Teleman.
\newblock {\em The {K}obayashi-{H}itchin correspondence}.
\newblock World Scientific Publishing Co., Inc., River Edge, NJ, 1995.

\bibitem{MR3745874}
Y.~Nie and X.~Zhang.
\newblock Semistable {H}iggs bundles over compact {G}auduchon manifolds.
\newblock {\em J. Geom. Anal.}, 28(1):627--642, 2018.

\bibitem{MR2815674}
C.~Okonek, M.~Schneider, and H.~Spindler.
\newblock {\em Vector bundles on complex projective spaces}.
\newblock Modern Birkh\"auser Classics. Birkh\"auser/Springer Basel AG, Basel,
  2011.
\newblock Corrected reprint of the 1988 edition, With an appendix by S. I.
  Gelfand.

\bibitem{MR2180378}
D.~Popovici.
\newblock A simple proof of a theorem by {U}hlenbeck and {Y}au.
\newblock {\em Math. Z.}, 250(4):855--872, 2005.

\bibitem{MR944577}
C.~T. Simpson.
\newblock Constructing variations of {H}odge structure using {Y}ang-{M}ills
  theory and applications to uniformization.
\newblock {\em J. Amer. Math. Soc.}, 1(4):867--918, 1988.

\bibitem{MR861491}
K.~Uhlenbeck and S.-T. Yau.
\newblock On the existence of {H}ermitian-{Y}ang-{M}ills connections in stable
  vector bundles.
\newblock volume~39, pages S257--S293. 1986.
\newblock Frontiers of the mathematical sciences: 1985 (New York, 1985).

\bibitem{MR4309984}
R.~Wang and P.~Zhang.
\newblock The {H}itchin-{K}obayashi correspondence for holomorphic pairs over
  non-{K}\"ahler manifolds.
\newblock {\em Bull. Sci. Math.}, 172:Paper No. 103050, 32, 2021.

\bibitem{MR4237961}
C.~Zhang, P.~Zhang, and X.~Zhang.
\newblock Higgs bundles over non-compact {G}auduchon manifolds.
\newblock {\em Trans. Amer. Math. Soc.}, 374(5):3735--3759, 2021.

\end{thebibliography}

\end{document}